\documentclass[12pt]{amsart}
\usepackage{etex}
\usepackage{graphicx}
\usepackage{enumerate}
\usepackage{empheq}
\usepackage{subfigure}
\usepackage{empheq}
\usepackage{amssymb}
\usepackage{ccaption}
\usepackage[margin=1in]{geometry}
\usepackage{setspace}
\usepackage{pst-all}
\usepackage{array}
\usepackage{comment}

\usepackage{booktabs,tikz,color,verbatim}
\usetikzlibrary{positioning, arrows,calc,shapes,decorations.pathreplacing,shadows}

\newlength{\cellsize} 
\newsavebox{\cell}

\sbox{\cell}{%
\begin{picture}(18,18)\linethickness{0.6pt} %
  \put(0,0){\line(1,0){18}} \put(0,0){\line(0,1){18}}
  \put(18,0){\line(0,1){18}} \put(0,18){\line(1,0){18}}
\end{picture}}

\newcommand\cellify[1]{%
  \def\thearg{#1}\def\nothing{}%
  \ifx\thearg\nothing \vrule width0pt height\cellsize depth0pt\else
  \hbox to 0pt{\usebox{\cell} \hss}\fi%
  \vbox to \cellsize{ \vss \hbox to
  \cellsize{\hss$#1$\hss} \vss}% <-- big-letter fillings
}
\newcommand\tableau[1]{\vtop{\let\\\cr
\baselineskip -16000pt \lineskiplimit 16000pt \lineskip 0pt
\ialign{&\cellify{##}\cr#1\crcr}}}

\newcommand\bas[1]{\omit \vbox to \cellsize{ \vss \hbox to \cellsize{\hss$#1$\hss} \vss}}
\captiondelim{}
\newfixedcaption{\outfigcaption}{figure}

\theoremstyle{plain}
\newtheorem{thm}{Theorem}

\newtheorem{lem}[thm]{Lemma}

\theoremstyle{definition}
\newtheorem{defi}[thm]{Definition}
\newtheorem{exa}[thm]{Example}

\newtheorem{rem}[thm]{Remark}

\newcommand{\la}{\lambda}

\newcommand{\spam}{\mathrm{span}}

\newcommand{\new}[1]{#1}

\title[Chromatic symmetric function bases]{Chromatic bases for symmetric functions}
\thanks{The first author was supported by Basic Science Research Program through the National Research Foundation of Korea (NRF) funded by the Ministry of Education (NRF-2015R1D1A1A01057476).    The second author was supported in part by the National Sciences and Engineering Research Council of Canada.}

\author{Soojin Cho}
\address{
 Department of Mathematics\\
 Ajou University\\
 Suwon  443-749, Korea}
\email{chosj@ajou.ac.kr}

\author{Stephanie van Willigenburg}
\address{
 Department of Mathematics \\
 University of British Columbia\\
 Vancouver BC V6T 1Z2, Canada}
\email{steph@math.ubc.ca}

\keywords{chromatic symmetric function, complete graph, star graph, path, cycle}

\subjclass[2010]{Primary 05E05; Secondary 05C15, 05C25}

\begin{document}

\begin{abstract}
In this note we obtain numerous new bases for the algebra of symmetric functions whose generators are chromatic symmetric functions. More precisely, if $\{ G_ k \} _{k\geq 1}$ is a set of connected graphs such that $G_k$ has $k$ vertices for each $k$, then the set of all chromatic symmetric functions $\{ X_{G_ k} \} _{k\geq 1}$ generates the algebra of symmetric functions.
 We also obtain explicit expressions for the generators arising from complete graphs, star graphs, path graphs and cycle graphs.
\end{abstract}
\maketitle
%\setcounter{section}{-1}
%%%%%%%%%%%%%%%%%%%%%%%%%%%%%%%%%%%%%%%%%%%%%%%%%%%%%%%%%%%%%%%%%%%%%%%%%
\section{Introduction} \label{sec:intro}
%%%%%%%%%%%%%%%%%%%%%%%%%%%%%%%%%%%%%%%%%%%%%%%%%%%%%%%%%%%%%%%%%%%%%%%%%
In \cite{Stan95} Stanley defined a symmetric function $X_G$ that was reliant on a finite simple graph $G$, called the chromatic symmetric function of $G$. He proved that $X_G$ specializes to the chromatic polynomial of $G$ and \new{generalizes} other chromatic polynomial properties, although intriguingly not the deletion-contraction property. Since then the chromatic symmetric function $X_G$ has been the genesis of two long-standing conjectures in algebraic combinatorics. The first of these conjectures that the chromatic symmetric functions of all $(3+1)$-free posets are a positive linear combination of elementary symmetric functions, for which a variety of evidence exists \cite{Gasharov, GPaquet, StanStem}. The second of these conjectures that the chromatic symmetric function distinguishes non-isomorphic trees. This conjecture has been confirmed for spiders \cite{MMW} plus a variety of caterpillars \cite{Jose2, MMW}, and towards a different approach a sufficient condition for graphs to have equal chromatic symmetric functions has also been discovered \cite{Orellana}.

In this vein of a different perspective on the chromatic symmetric function  we provide a potential new tool, namely we determine a myriad of new bases for the algebra of symmetric functions, whose generators are chromatic symmetric functions (Theorem~\ref{thm:bases}) and give explicit expansions for the generators arising from well-known graphs (Theorem~\ref{thm:expansions}). 
%%%%%%%%%%%%%%%%%%%%%%%%%%%%%%%%%%%%%%%%%%%%%%%%%%%%%%%%%%%%%%%%%%%%%%%%%
\section{Chromatic symmetric function bases}\label{sec:bases}
%%%%%%%%%%%%%%%%%%%%%%%%%%%%%%%%%%%%%%%%%%%%%%%%%%%%%%%%%%%%%%%%%%%%%%%%%
We begin by recalling concepts that will be useful later, and by defining the algebra of symmetric functions. A \emph{partition} $\lambda = (\lambda _1, \ldots , \lambda _\ell)$ of $n$, denoted by $\lambda \vdash n$, is a list of positive integers whose parts $\lambda _i$ satisfy $\lambda _1 \geq \cdots \geq \lambda _\ell$ and $\sum _{i=1}^\ell \lambda _i = n$. If $\lambda$ has exactly $m_i$ parts equal to $i$ for $1\leq i \leq n$ we will sometimes denote $\lambda$ by $\lambda = (1^{m_1},  \ldots , n^{m_n})$. Also, given  partitions of $n$, $\lambda = (\lambda _1, \ldots , \lambda _\ell)$ and $\mu = (\mu _1, \ldots , \mu _m)$ we say $\mu  \leq \lambda$ \emph{in lexicographic order} if $\mu = \lambda$ or $\mu _i = \lambda _i$ for $i<j$ and $\mu_j< \lambda _j$ for some $1\leq j\leq m$.

The algebra of symmetric functions is a subalgebra of $\mathbb{Q} [[x_1, x_2, \ldots ]]$ and can be defined as follows. We define the \emph{$i$-th power sum symmetric function} $p_i, i\geq 1$ to be
$$p_i = x^i_1 + x^i_2 + x^i_3 + \cdots$$and given a partition $\lambda = (\lambda _1, \ldots , \lambda _\ell)$, we define the \emph{power sum symmetric function} $p_\lambda$ to be
$$p_\lambda = p_{\lambda_1}\cdots p_{\lambda _\ell}\,.$$Then the \emph{algebra of symmetric functions}, $\Lambda$, is the graded algebra
$$\Lambda = \Lambda ^0 \oplus \Lambda ^1 \oplus \cdots$$where $\Lambda ^0 = \spam \{ 1 \} = \mathbb{Q}$ and for $n\geq 1$
$$\Lambda ^n = \spam \{ p_\lambda \,|\, \lambda \vdash n\}\,.$$The power sum symmetric functions in fact form a basis for $\Lambda$. Other well-known bases include the basis of Schur functions, the basis of complete homogeneous symmetric functions and  the basis of elementary symmetric functions, whose \emph{$i$-th elementary symmetric function} $e_i, i\geq 1$ is defined to be
$$e_i = \sum _{j_1< \cdots < j _i} x_{j_1} \cdots x_{j_i}$$leading to the celebrated fundamental theorem of symmetric functions, which states that
$$\Lambda = \mathbb{Q} [e_1, e_2, \ldots ]\,.$$

Our object of study is a further symmetric function, known as the chromatic symmetric function, which is reliant on a graph that is \emph{finite} and  \emph{simple}, and from here onwards we will assume that all our graphs satisfy these properties. We are now almost ready to define the chromatic symmetric function of  a graph, but before we do we recall the notion of a proper coloring. Given a graph $G$ with vertex set $V$  a \emph{proper coloring} $\kappa$ of $G$ is a function $$\kappa : V \rightarrow \{ 1, 2, \ldots \}$$such that if $v_1, v_2 \in V$ are adjacent, then $\kappa (v_1) \neq \kappa (v_2)$.

\begin{defi}\label{defi:XG} For a graph $G$ with vertex set $V=\{v_1, \dots, v_n\}$ and edge set $E$
 the \emph{chromatic symmetric function of $G$} is defined to be
$$ X_G=\sum_{\kappa} x_{\kappa(v_1)}   \cdots x_{\kappa(v_n)}$$where the sum is over all proper colorings $\kappa$ of $G$.\end{defi}

Given a graph $G$  with vertex set $V=\{v_1, \dots, v_n\}$ and edge set $E$, and a subset $S\subseteq E$, let $\la(S)$ be the partition of $n$ whose parts are equal to the number of vertices in the connected components of the spanning subgraph of $G$ with vertex set $V$ and edge set $S$.
We say a set partition $\pi=\{B_1,  \ldots, B_k\}$ of $V$ is \emph{connected} if the subgraph of $G$ determined by $B_i$ is connected for each $i$, and
 the \emph{lattice of contractions $L_G$  of $G$} is the set of all connected partitions of $V$ partially ordered by refinement \new{so that the unique minimal element $\hat{0}$ of $L_G$ is the partition into $n$ one element blocks}. Lastly, given  $\pi=\{B_1,  \ldots, B_k\}\in L_G$, the \emph{type}  of $\pi$, denoted by $\rm{type}(\pi)$, is the partition obtained by rearranging  $|B_1|,  \ldots, |B_k|$  in weakly decreasing order.  \new{With} this in mind we have the following.

\begin{lem}\label{lem:pbasis}\cite[Theorems 2.5 and 2.6]{Stan95} For a graph $G$ with vertex set $V$ and  edge set $E$ we have that
\begin{enumerate}[1.]
\item $X_G=\sum_{S\subseteq E} (-1)^{|S|} p_{\la(S)}$,
\item $X_G=\sum_{\pi\in L_G} \mu(\hat{0}, \pi) p_{\rm{type}(\pi)}\,,$ \new{where $\mu$ is the M\"{o}bius function of $L_G$}, and $\mu(\hat{0}, \pi)$ is non-zero for all $\pi\in L_G$.
\end{enumerate}
\end{lem}

The chromatic symmetric function also satisfies the  following useful property.

\begin{lem}\cite[Proposition 2.3]{Stan95}\label{lem:disjoint} If a graph $G$ is a disjoint union of \new{subgraphs} $G_1, \dots, G_\ell$, then $X_G=\prod_{i=1}^\ell X_{G_i}$.
\end{lem}

We now need one last definition before we can state our theorem.

\begin{defi}\label{defi:G}
Let $\{ G_ k \} _{k\geq 1}$ be a set of connected graphs such that $G_k$ has $k$ vertices for each $k$, and let $\lambda = (\lambda _1 , \ldots , \lambda _\ell)$ be a partition. Then
$$G_\lambda = G _{\lambda _1}\cup \cdots \cup G _{\lambda _\ell}\,,$$that is, $G_\lambda$ is the graph whose connected components are $G_{\lambda _1}, \ldots , G_{\lambda_\ell}$.
\end{defi}

We can now determine a plethora of new bases for $\Lambda$.

\begin{thm}\label{thm:bases}
Let $\{ G_ k \} _{k\geq 1}$ be a set of connected graphs such that $G_k$ has $k$ vertices for each $k$. Then
$$\{ X_{G_\lambda} \,|\, \lambda \vdash n\}$$is a $\mathbb{Q}$-basis of $\Lambda ^n$. Plus we have  that
$$\Lambda = \mathbb{Q} [X_{G_1}, X_{G_2}, \ldots ]$$and the $X_{G_k}$ are algebraically independent over $\mathbb{Q}$.
\end{thm}

\begin{proof} Let $\la=(\la_1, \dots, \la_\ell)\vdash n$ and $V_i$ be the sets of vertices in $G_{\la_i}$ for $1\leq i\leq \ell$. Then $$V=\biguplus_{i=1}^\ell V_i$$is the set of vertices in $G _\la$. By the definition of $G_\la$, we know that if $\pi \in L_{G_\la}$, then $\rm{type}(\pi)\leq \lambda$ in lexicographic order.
Thus by Lemma~\ref{lem:pbasis} it follows that $$X_{G_\la}=\sum_{\mu\leq \la} c_{\la\mu} p_\mu$$and, moreover, that $c_{\la\la}=\mu(\hat{0}, \pi_\la)\neq 0$ where $\pi_\la=(V_1,  \ldots, V_\ell)$ is the unique connected partition of $V$ satisfying $\rm{type}(\pi_\la)=\la$. Hence, $\{ X_{G_\lambda} \,|\, \lambda \vdash n\}$ is a $\mathbb{Q}$-basis of $\Lambda ^n$.

Since for $\la=(\la_1, \dots, \la_\ell)$ we have
\begin{equation}\label{eq:Glambda}
X_{G_\la}=\prod_{i=1}^\ell X_{G_{\la _i}}
\end{equation}by Lemma~\ref{lem:disjoint} and $\{ X_{G_\lambda}\} _{\lambda \vdash n\geq 1}\cup \{1\}$ forms a $\mathbb{Q}$-basis for $\Lambda$, every element of $\Lambda$ is expressible uniquely as a polynomial in the $X_{G_k}$ and hence  $\Lambda = \mathbb{Q} [X_{G_1}, X_{G_2}, \ldots ]$ and the $X_{G_k}$ are algebraically independent over $\mathbb{Q}$.
\end{proof}
 
\pagebreak
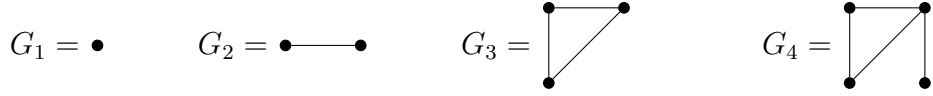
\begin{exa}\label{exa:1}

If 
%\begin{comment}
\begin{figure}[!ht]
  \begin{tikzpicture}
    \tikzstyle{Element} = [circle, draw, fill=black!100,
                        inner sep=0pt, minimum width=4pt]
    \node at (2.0,3.0) {$G_1=$};
    
        \node[Element] at (2.7,3.0) {} ;
             
        \node at  (4.5, 3.0) {$G_2=$};
        \node[Element] (a) at (5.2, 3.0) {};
        \node[Element] (b) at (6.2, 3.0) {};
        \draw (a) to (b);        
        
        \node at  (8.0, 3.0) {$G_3=$};
        \node[Element] (c) at (8.7, 2.5) {};
        \node[Element] (d) at (8.7, 3.5) {};
        \node[Element] (e) at (9.7, 3.5) {};
        \draw (c) to (d);
        \draw (d) to (e);
        \draw (e) to (c);

        \node at  (12.0, 3.0) {$G_4=$};
        \node[Element] (c) at (12.7, 2.5) {};
        \node[Element] (d) at (12.7, 3.5) {};
        \node[Element] (e) at (13.7, 3.5) {};
        \node[Element] (f) at (13.7, 2.5) {};
        \draw (c) to (d);
        \draw (d) to (e);
        \draw (e) to (c);
        \draw (e) to (f);
\end{tikzpicture}
\end{figure}
%\end{comment}

\noindent then $\{X_{G_1}, X_{G_2}, X_{G_3}, X_{G_4}\}$ is a set of generators for $\Lambda^4$ and

$X_{G_{(4)}}=X_{G_4}=-2p_{(4)}+4p_{(3,1)}+p_{(2,2)}-4p_{(2, 1, 1)}+p_{(1, 1, 1, 1)}$\,,

$X_{G_{(3,1)}}=X_{G_3}X_{G_1}=(2p_{(3)}-3p_{(2,1)}+p_{(1, 1, 1)})p_{(1)}=2p_{(3,1)}-3p_{(2,1,1)}+p_{(1, 1, 1, 1)}$\,,

$X_{G_{(2,2)}}=X_{G_2}X_{G_2}=(-p_{(2)}+p_{(1, 1)})^2=p_{(2, 2)}-2p_{(2, 1, 1)}+p_{(1, 1, 1, 1)}$\,,

$X_{G_{(2,1, 1)}}=X_{G_2}X_{G_1}X_{G_1}=(-p_{(2)}+p_{(1, 1)})p_{(1)}p_{(1)}=-p_{(2, 1, 1)}+p_{(1, 1, 1, 1)}$\,, and

$X_{G_{(1, 1,1, 1)}}=X_{G_1}X_{G_1}X_{G_1}X_{G_1}=p_{(1)}^4=p_{(1, 1, 1, 1)}$

\noindent is a $\mathbb{Q}$-basis of $\Lambda^4$. 

 Alternatively, if $A, B, C, D$ are as below, then $\{ X_A, X_B, X_C, X_D\}$ is a set of generators for $\Lambda^4$ and $$\{X_D, X_CX_A, X_BX_B, X_BX_AX_A, X_AX_AX_AX_A\}$$is a $\mathbb{Q}$-basis of $\Lambda^4$.
 
%\begin{comment}
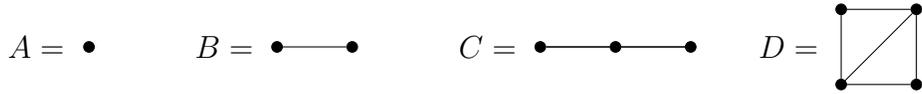
\begin{figure}[!ht]
  \begin{tikzpicture}
    \tikzstyle{Element} = [circle, draw, fill=black!100,
                        inner sep=0pt, minimum width=4pt]
    \node at (2.0,3.0) {$A=$};
    
        \node[Element] at (2.7,3.0) {} ;
             
        \node at  (4.5, 3.0) {$B=$};
        \node[Element] (a) at (5.2, 3.0) {};
        \node[Element] (b) at (6.2, 3.0) {};
        \draw (a) to (b);        
        
        \node at  (8.0, 3.0) {$C=$};
        \node[Element] (c) at (8.7, 3) {};
        \node[Element] (d) at (9.7, 3) {};
        \node[Element] (e) at (10.7, 3) {};
        \draw (c) to (d);
        \draw (d) to (e);
        \draw (e) to (c);

        \node at  (12.0, 3.0) {$D=$};
        \node[Element] (c) at (12.7, 2.5) {};
        \node[Element] (d) at (12.7, 3.5) {};
        \node[Element] (e) at (13.7, 3.5) {};
        \node[Element] (f) at (13.7, 2.5) {};
        \draw (c) to (d);
        \draw (d) to (e);
        \draw (e) to (c);
        \draw (e) to (f);
        \draw (c) to (f);
\end{tikzpicture}
\end{figure}
%\end{comment}

\end{exa}

\begin{rem}\label{rem:Qbasis}
Observe that the only connected graph on two vertices is $G_2$ above, and
$$e_2= \sum _{j_1<j_2} x _{j_1}x_{j_2} = \frac{1}{2} X_{G_2}\,.$$Therefore, while every $\{ X_{G_\lambda}\} _{\lambda \vdash n\geq 1}\cup \{1\}$ is a $\mathbb{Q}$-basis of $\Lambda$ it is never a $\mathbb{Z}$-basis of $\Lambda$.
\end{rem}

%%%%%%%%%%%%%%%%%%%%%%%%%%%%%%%%%%%%%%%%%%%%%%%%%%%%%%%%%%%%%%%%%%%%%%%%%
\section{Chromatic symmetric functions for classes of graphs}\label{sec:special case}
%%%%%%%%%%%%%%%%%%%%%%%%%%%%%%%%%%%%%%%%%%%%%%%%%%%%%%%%%%%%%%%%%%%%%%%%%

In this section, we compute chromatic symmetric functions for some particular connected graphs, whose definitions we include for clarity. The \emph{complete graph} $K_n, n\geq 1$ has $n$ vertices each pair of which are adjacent. The \emph{star graph} $S_{n+1}, (n+1)\geq 1$ has $n+1$ vertices and is the tree with one vertex of degree $n$ and $n$ vertices of degree one. The \emph{path graph} $P_n, n\geq 1$ has $n$ vertices and is the tree with $2$ vertices of degree one and $n-2$ vertices of degree $2$ for $n\geq 2$ and $P_1=K_1$. Lastly, the \emph{cycle graph}  $C_n, n\geq 1$ is the connected graph with $n$ vertices of degree $2$ for $n\geq 3$, $C_2=K_2$, $C_1=K_1$. The chromatic symmetric functions of Ferrers graphs, naturally related to $\Lambda$ via Ferrers diagrams, were computed in \cite{EvW}. We note that the third formula appears in the second proof of \cite[Proposition 5.3]{Stan95} that gives the generating function for $X_{P_n}$, and the generating function for $X_{C_n}$ is given in  \cite[Proposition 5.4]{Stan95}.

\begin{thm}\label{thm:expansions}
\begin{enumerate}[1.]
\item If $K_n$ is the complete graph with $n\geq 1$ vertices, then $$X_{K_n}=n!e_n\,.$$

\item If $S_{n+1}$ is the star graph with $(n+1) \geq 1$ vertices, then 
$$X_{S_{n+1}}=\sum_{r=0}^n (-1)^{r} {n\choose r} p_{(r+1, 1^{n-r})}\,.$$  

\item If $P_n$ is the path graph with $n\geq 1$ vertices, then
$$X_{P_n}=\sum_{\la=(1^{m_1}, \ldots, n^{m_n})\vdash n}(-1)^{n-\sum_{i=1}^n m_i}\frac{(\sum_{i=1}^n m_i)\, !}{\prod_{i=1}^n (m_i)!}\,\, p_{\la}.$$

\item If $C_n$ is the cycle graph with $n\geq 1$ vertices, then
$$X_{C_n}=\sum_{\la=(1^{m_1},  \dots, n^{m_n})\vdash n}(-1)^{n-\sum_{i=1}^n m_i}\frac{(\sum_{i=1}^n m_i) \, !}{\prod_{i=1}^n (m_i)!}\left(1+\sum_{j=2}^n (j-1) \frac{m_j}{\sum_{i=1}^n m_i} \right) p_\la +(-1)^np_n\,.$$
\end{enumerate}
\end{thm}

\begin{proof}
For $K_n$, since every vertex must be colored a different color and this can be done in $n!$ ways, $X_{K_n} = n! \sum _{j_1 < \cdots < j_n} x_{j_1} \cdots x_{j_n} = n! e_n$.

We now use the first part of Lemma~\ref{lem:pbasis}, which states that
$X_G=\sum_{S\subseteq E} (-1)^{|S|} p_{\la(S)}$, for the remainder of the proof, where $E$ is the set of edges of our graph $G$.

For $S_{n+1}$, if we choose any $r$ edges, then $(r+1)$ vertices will make a connected component and the remaining $(n-r)$ vertices will be isolated. Hence the second part of the theorem follows.

For $P_n$, draw this graph on a horizontal axis. Now consider the spanning subgraph of $P_n$ with $n$ vertices and edge set $S\subseteq E$, $P^S _n$. Counting the number of vertices in each connected component of $P^S _n$ from left to right yields a list of positive integers when rearranged into weakly decreasing order yield a partition, say $\lambda = (1^{m_1},  \ldots , n^{m_n}) \vdash n$. Since the number of edge sets, $S$, that will yield $\lambda$ is $\frac{(\sum_{i=1}^n m_i)\, !}{\prod_{i=1}^n (m_i)!}$, \new{and $|S|=\sum_{i=1}^n (i-1)m_i = n-\sum_{i=1}^n m_i$ for such $S$}, the third part is now proved.

For $C_n$ and a partition $\la=(1^{m_1},  \dots, n^{m_n})\vdash n$, we look for all subsets $S$ of the edge set $E$ that contribute to $p_\la$ in the expansion of $X_{C_n}$; that is, subsets $S$ satisfying $\la(S)=\la$. To this end, label the vertices of $C_n$ with $v_1, \dots, v_n$ in a clockwise direction, choosing $v_1$ arbitrarily, and let $\epsilon_i$ be the edge connecting $v_i$ and $v_{i+1}$ for $i=1, \dots, n$, where $v_{n+1}=v_1$. We first consider the possible $S$ that do not contain $\epsilon_n$: Since $\epsilon_n\not \in S$, such $S$ can be understood as a subset of the vertex set of $P_n$, and the contribution of such $S$ to $p_\la$ in the expansion of $X_{C_n}$ is the same as the coefficient of $p_\la$ in the expansion of $X_{P_n}$, which is 
\begin{equation}\label{eq:1}
(-1)^{n-\sum_{i=1}^n m_i}\frac{(\sum_{i=1}^n m_i)\, !}{\prod_{i=1}^n (m_i)!}\,.
\end{equation}
\newline We now consider the possible $S\neq E$ that do contain $\epsilon_n$: There are cases to consider, depending on the number of vertices $j$, where $j\geq 2$, in the connected component of $S$ that contains $\epsilon_n$. For each $j$, there are $(j-1)$ possible connected components depending on the smallest labeled vertex of the component, which can be $v_{n-j+2}, \ldots , v_n$. After we identify the connected component containing $\epsilon_n$ in $C_n$, the remainder of the graph is the path graph with $(n-j)$ vertices to which we can apply the third part of this theorem. The overall contribution of such $S$ is thus
\begin{equation}\label{eq:2}
\sum_{j=2}^n (-1)^{j-1}(j-1)(-1)^{n-j-(\sum_{i=1}^{n} m_i)+1}\frac{((\sum_{i=1}^n m_i)-1)\, ! (m_j)}{\prod_{i=1}^n (m_i)!}.
%& &=\sum_{j=2}^n(-1)^{n-\sum_{i=1}^{n} m_i}(j-1)\frac{((\sum_{i=1}^n m_i)-1)\, ! (m_j)}{\prod_{i=1}^n (m_i)!}
\end{equation}
We now add (\ref{eq:1}) and (\ref{eq:2}) to obtain that the coefficient of $p_\lambda$ is $$(-1)^{n-\sum_{i=1}^n m_i}\frac{(\sum_{i=1}^n m_i) \, !}{\prod_{i=1}^n (m_i)!}\left(1+\sum_{j=2}^n (j-1) \frac{m_j}{\sum_{i=1}^n m_i} \right)\,.$$Finally, when $S=E$ we obtain the term $(-1)^np_n$.
\end{proof}

\begin{rem}\label{rem:schurpositive} A natural question to ask is whether any of these bases is Schur positive, that is, a positive linear combination of Schur functions. The answer is not at all obvious, since the basis whose generators stem from complete graphs is trivially Schur positive, whereas the basis whose generators stem from star graphs is not as $X_{S_4}$ is not.
\end{rem}

%%%%%%%%%%%%%%%%%%%%%%%%%%%%%%%%%%%%%%%%%%%%%%%%%%%%%%%%%%%%%%%%%%%%%%%%%%%%
%%%%%%%%%%%%%%%%%%%%%%%%%%%%%%%%%%%%%%%%%%%%%%%%%%%%%%%%%%%%%%%%%%%%%%%%%%%%

\section*{Acknowledgements}\label{sec:acknow} The authors would like to thank Samantha Dahlberg and Boram Park for helpful conversations, the referee for helpful comments, and Ajou University where the research took place.

\bibliographystyle{amsplain}

\end{document}